\documentclass[12pt]{amsart}
\usepackage{amsmath,amsfonts,amsthm,amscd,textcomp,times, amssymb}
\usepackage[all,cmtip]{xy}
\usepackage[T1]{fontenc}
\usepackage{calligra}
\usepackage{multido}
\usepackage{frcursive}
\usepackage{mathrsfs}
\usepackage{tikz}
\usetikzlibrary{shapes,arrows}
\newtheorem{theorem}{Theorem}

\newtheorem{proposition}{Proposition}

\newtheorem{lemma}{Lemma}
\newtheorem{definition}{Definition}
\newtheorem{example}{Example}
\newtheorem{corollary}{Corollary}
\newtheorem{remark}{Remark}
\numberwithin{equation}{section}
\numberwithin{remark}{section}
\numberwithin{theorem}{section}
\numberwithin{proposition}{section}
\numberwithin{definition}{section}
\numberwithin{lemma}{section}
\numberwithin{claim}{section}
\numberwithin{corollary}{section}
\numberwithin{conjecture}{section}

\newcommand{\bull}{\ensuremath{{}\bullet{}}}
 
\newcommand{\cpn}{\ensuremath{\mathbb{P}^{N}}}
\newcommand{\slnc}{\ensuremath{SL(N+1,\mathbb{C})}}

\newcommand{\ra}{\ensuremath{\longrightarrow}}
\newcommand{\om}{\ensuremath{\omega}}

\newcommand{\vps}{\ensuremath{\varphi_{\sigma}}}

\newcommand{\cn}{\ensuremath{\mathbb{C}^{N+1}}}

\newcommand{\xhyp}{\ensuremath{X\times\mathbb{P}^{n-1}}}

\newcommand{\elam}{\ensuremath{\mathbb{E}_{\lambda_{\bull}}}}
\newcommand{\emu}{\ensuremath{\mathbb{E}_{\mu_{\bull}}}}

\begin{document}
\bibliographystyle{alpha}
%%%%%%%%%%%%%%%%%%%%%%%%%%%%%%%%%%%%%%%%%%%%%%%%%%%%%%%%%%%%%%%%%%%%%%%
\title[numerical criterion ]{A numerical criterion for K-energy maps of Algebraic manifolds $^\dagger$ }
\author{Sean Timothy Paul}   
\thanks{$\dagger$ The author is supported in part by NSF DMS grant no.1104448}
\email{stpaul@math.wisc.edu}
\address{l'Institut Henri Poincar\'e, 11 Rue Pierre et Marie Curie, 75005 Paris, France}
\address{and the Mathematics Dept. Univ. of Wisconsin, Madison}
 %%%%%%%%%%%%%%%%%%%%%%%%%%%%%%%%%%%%%%%%%%%%%%%%%%%%%%%%%%%%%%%%%%%%%%%
\date{October 2, 2012}
 \vspace{-5mm}
\begin{abstract} {Let $X\ra\cpn$ be a smooth, linearly normal, complex projective variety.  Let $\mathcal{B}$ denote the corresponding space of Bergman metrics. In this paper it is shown that the Mabuchi energy of $(X,{\om_{FS}}|_X)$ is bounded below on all algebraic degenerations in $\mathcal{B}$ if and only if it is \emph{uniformly} bounded below on $\mathcal{B}$ .}
   \end{abstract}
\maketitle
\setcounter{tocdepth}{1}

 \section{Statement of the Main Theorem  }
 Let $X\ra \cpn$ be a smooth subvariety embedded by a very ample complete linear system. Let $\om_{FS}$ denote the Fubini Study metric on $\cpn$ induced by a choice of positive definite Hermitian form on $\cn$.  As usual $\nu_{\om}$ denotes the K-energy map (see \cite{mabuchi} for the definition of the K-energy map )  of the restriction  $\om:={\om_{FS}}|_X$ . Let $\mathcal{B}$ denote the Bergman metrics associated to this embedding. The main result of this paper is the following.
 \begin{theorem}\label{mainthm}
 Assume that for every degeneration $\lambda$ in $\mathcal{B}$ there is a (finite) constant $C(\lambda)$ such that
 \begin{align*}
 \lim_{\alpha\ra 0}\nu_{\om}(\varphi_{\lambda(\alpha)})\geq  C(\lambda) \ .
 \end{align*}
 Then there is a \textbf{uniform} constant $C$ such that for all $\vps\in \mathcal{B}$ we have the lower bound
 \begin{align*}
 \nu_{\om}(\vps)\geq C   \ .
 \end{align*}
 \end{theorem}
 Theorem \ref{mainthm} reduces the problem of obtaining a lower bound for the Mabuchi energy restricted to $\mathcal{B}$ to the problem of checking for a bound on each algebraic one parameter subgroup in $\mathcal{B}$. This latter problem has a ``polyhedral-combinatorial'' characterization (see Corollary \ref{polytopes}) in terms of  the weight polyhedra of the $X$-\emph{resultant} and $X$- \emph{hyperdiscriminant} of Cayley and Gelfand, Kapranov, and Zelevinsky \cite{gkz}.

%%%%%%%%%%%%%%%%%%%%%%%%%%%%%%%%%%%%%%%%%%%%%%%%%%%%%%%%%%%%%%%%%%%%%%%%%%%%
\section{semistable pairs}
Let $G$ be one of the classical subgroups of $GL(N+1,\mathbb{C})$, for example the special linear group.
 Let $(\mathbb{V},\rho)$ be a finite dimensional complex rational representation of $G$. Recall that $(\mathbb{V},\rho)$ is  {rational} provided that for all $\alpha\in \mathbb{V}^{\vee}$ (dual space) and $v\in \mathbb{V}\setminus \{0\}$ the  {matrix coefficient}
\begin{align}
\varphi_{\alpha , v}:G\ra \mathbb{C} \quad \varphi_{\alpha , v}(\sigma):=\alpha(\rho(\sigma)\cdot v) 
\end{align}
is a {regular function} on $G$, that is
\begin{align}
\varphi_{\alpha , v}\in \mathbb{C}[G]:= \mbox{affine coordinate ring of $G$} .
\end{align}

For any vector space $\mathbb{V}$ and any $v\in \mathbb{V}\setminus\{0\}$ we  let $[v]\in\mathbb{P}(\mathbb{V})$ denote the line through $v$. If $\mathbb{V}$ is a $G$ module define $\mathcal{O}_{v}:=G\cdot [v]\subset \mathbb{P}(\mathbb{V})$ the projective orbit of $[v]$ .  We let $\overline{\mathcal{O}}_{v}$ denote the Zariski closure of this orbit. Given a pair  $(v\in\mathbb{V}\setminus\{0\} \ ; \ w\in \mathbb{W}\setminus\{0\} )$ we consider the orbits inside the projective space of the direct sum  
\begin{align}
\mathcal{O}_{vw}:=G\cdot [(v,w)]  \subset \mathbb{P}(\mathbb{V}\oplus\mathbb{W}) \ , \ \mathcal{O}_{v} \subset \mathbb{P}(\mathbb{V}\oplus\{0\})\ .
\end{align}
  Our basic definition is the following.
  \begin{definition}\label{pair}
\emph{The pair $(v,w)$ is \textbf{\emph{semistable}} if and only if}  $ \overline{\mathcal{O}}_{vw}\cap\overline{\mathcal{O}}_{v}=\emptyset $ .
  \end{definition} 
\begin{example}\label{hmss} \emph{Let $\mathbb{V}=\mathbb{C}, v=1$ (the trivial one dimensional representation). Let $\mathbb{W}$ be any $G$ module.
Then  the pair $(1,w)$ is semistable if and only if
\begin{align}
0\notin \overline{G\cdot w}\subset \mathbb{W} \quad (\mbox{\emph{affine} closure} ) \ .
\end{align}
 In other words ,  $w$ is semistable in the usual (Hilbert-Mumford) sense.}
\end{example}

\subsection{Numerical Semistability}
Let $H$ denote any maximal algebraic torus of $G$.  Let $M_{\mathbb{Z}}=M_{\mathbb{Z}}(H)$ denote the {character lattice} of $H$
\begin{align}
M_{\mathbb{Z}}:= \mbox{Hom}_{\mathbb{Z}}(H,\mathbb{C}^*) \ . 
\end{align}

As usual, the dual lattice is denoted by $N_{\mathbb{Z}}$. It is well known that $ u\in N_{\mathbb{Z}}$ corresponds to an algebraic one parameter subgroup $\lambda^u$ of $H$. These are algebraic homomorphims $\lambda:\mathbb{C}^*\ra H$.  The correspondence is given by 
\begin{align}
(\cdot \ , \ \cdot) :N_{\mathbb{Z}}\times M_{\mathbb{Z}}\ra \mathbb{Z} \ , \ m(\lambda^u(t))=t^{(u , m)} \ .
\end{align}
 
 We introduce associated real vector spaces by extending scalars 
\begin{align} 
 \begin{split}
 &M_{\mathbb{R}}:= M_{\mathbb{Z}}\otimes_{\mathbb{Z}}\mathbb{R} \\
\ \\
& N_{\mathbb{R}}:= N_{\mathbb{Z}}\otimes_{\mathbb{Z}}\mathbb{R}\ .
\end{split}
\end{align}

Since $\mathbb{V}$ is rational it decomposes under the action of $H$ into  {weight spaces}
\begin{align}
\begin{split}
&\mathbb{V}=\bigoplus_{a\in {\mathscr{A}}}\mathbb{V}_{a}  \\
\ \\
& \mathbb{V}_{a}:=\{v\in \mathbb{V}\ |\ h\cdot v=a(h) v \ , \ h\in H\}
\end{split}
\end{align}
$\mathscr{A}$ denotes  the {support} of $\mathbb{V}$ 
\begin{align}
\mathscr{A}:= \{a \in M_{\mathbb{Z}}\ | \ \mathbb{V}_{a}\neq 0\} \ .
\end{align}
Given $v\in \mathbb{V}\setminus \{0\}$  the projection of $v$ into $\mathbb{V}_{a}$ is denoted by $v_{a}$. The support of any (nonzero) vector $v$ is then defined by
\begin{align}
\mathscr{A}(v):= \{a\in \mathscr{A}\ | \ v_{a}\neq 0\} \ .
\end{align}
\begin{definition} \emph{ Let $H$ be any maximal torus in $G$. Let $v\in \mathbb{V}\setminus\{0\}$ . The \textbf{\emph{weight polytope}} of $v$ is the compact convex lattice polytope $\mathcal{N}(v)$ given by}
\begin{align}\label{wtpolytope}
\mathcal{N}(v):=\mbox{\emph{convex hull}}\ \mathscr{A}(v) \subset M_{\mathbb{R}} \ .
\end{align}
\end{definition}

\begin{definition}
\emph{The \textbf{\emph{weight}}  $w_{\lambda}(v)$  of $\lambda$ on $v\in \mathbb{V}\setminus \{0\}$ is the integer}
\begin{align}
\begin{split}
&w_{\lambda}(v):= \mbox{\emph{min}}_{  x\in \mathcal{N}(v) }\ {u}(x)= \mbox{\emph{min}} \{ (a,u )\ |\  a \in \mathscr{A}(v)\}\ .
\ \\
& \lambda \sim u\in N_{\mathbb{Z}} \ .
\end{split}
\end{align}
\end{definition}
Alternatively observe that the weight of $\lambda$ on $v$ is the unique integer $w_{\lambda}(v)$ such that the limit
\begin{align}
\lim_{\alpha\ra 0}\alpha^{-w_{\lambda}(v)}\lambda(\alpha)\cdot v
\end{align}
exists in $\mathbb{V}$ and is not zero.

%%%%%%%%%%%%%%%%%%%%%%%%%%%%%%%%%%%%%%%%%%%%%%%%%%%%%%%%%%%%%%%%%%%%%%%%%%%% 

\begin{definition}\label{numerical}   \emph{Let $\mathbb{V},\mathbb{W}$ be $G$-modules. The pair $(v\in\mathbb{V}\setminus\{0\} \ , \ w\in \mathbb{W}\setminus\{0\} )$ is \textbf{\emph{numerically semistable}} if and only if  $\mathcal{N}(v)\subset \mathcal{N}(w)$    for all maximal algebraic tori $H\leq G$. Equivalently if and only if
\begin{align}
w_{\lambda}(w)\leq w_{\lambda}(v) \quad \mbox{for all one parameter subgroups $\lambda$ of $G$ . } \ .
\end{align}
  }
 \end{definition}

 \begin{proposition} (Paul \cite{paulcm2012}) \emph{Semistability implies numerical semistability}.
\begin{align}
\overline{\mathcal{O}}_{vw}\cap\overline{\mathcal{O}}_{v}=\emptyset \Rightarrow  \mathcal{N}(v)\subset \mathcal{N}(w) \ \mbox{ \emph{for all maximal algebraic tori} $H\leq G$ \ .   } 
\end{align}
\end{proposition}
 The crucial result of this section is the following.
  \begin{theorem}\label{numericalcriterion} \emph{ The pair $(v,w)$  is numerically semistable if and only if $\overline{\mathcal{O}}_{vw}\cap\overline{\mathcal{O}}_{v}=\emptyset$ .}
 \end{theorem}
The point of this result is that we would like to actually \emph{check} whether or not a given $(v,w)$ is semistable, that is, we would like to know if the Zariski closures of the orbits $\mathcal{O}_{vw}$ and $\mathcal{O}_v$ are disjoint. Theorem \ref{numericalcriterion} converts this problem into a ``polyhedral-combinatorial'' problem which is much easier to solve.  An application of Theorem \ref{numericalcriterion} generalizing David Hilbert's famous example of binary forms is the following.
   \begin{example}
\emph{Let $\mathbb{V}_e$ and  $\mathbb{V}_d$ be irreducible $SL(2,\mathbb{C})$ modules with highest weights $e , d \in\mathbb{N}$. These are well known to be spaces of homogeneous polynomials in two variables. Let $f$ and $g$ be two such polynomials in $\mathbb{V}_e\setminus\{0\}$ and $\mathbb{W}_d\setminus\{0\}$
respectively. Then the pair $(f,g)$ is semistable if and only if
\begin{align}\label{d-e/2}
e\leq d \ \mbox{and for all $p\in \mathbb{P}^1$}\ \mbox{ord}_p(g)-\mbox{ord}_p(f)\leq \frac{d-e}{2} \ .
\end{align}}
\end{example} 
   
   Theorem \ref{numericalcriterion} is closely related to the following condition, \emph{Property (P)}, which may or may not hold for a complex linear algebraic group $G$. Observe that this property is the projective version of  Property (A) from \cite{birkes71}.

\ \\   
\noindent \textbf{Property (P) .}  \emph{ If $\rho:G\ra GL(\mathbb{W})$ is a finite dimensional complex rational representation of $G$, and if $\mathscr{Z}$ is a non-empty $G$ invariant closed subvariety which meets $\overline{\mathcal{O}}_w\setminus {\mathcal{O}}_w$ , then there exists an element $z_0$ of  $\mathscr{Z}$ and an algebraic one-parameter subgroup $\lambda:\mathbb{C}^*\ra G$ such that}
\begin{align}
\lim_{\alpha\ra 0}\lambda(\alpha)\cdot [w]=z_0 \ .
\end{align}
If Property (P) holds for reductive groups then Theorem \ref{numericalcriterion} follows at once. 
\begin{remark}
\emph{It is known that when $G=T$ (an algebraic torus) Property (P) holds by Sun's Lemma (see \cite{paul2011}) .}
\end{remark}

We will show that under an assumption, which the author feels is not really necessary,  that Property (P) is true . The assumptions are easily seen to hold for the case at hand, namely when $\mathscr{Z}$ is the intersection  with $\overline{\mathcal{O}}_w$ of a $G$ invariant \emph{linear} subspace of $\mathbb{P}(\mathbb{W})$.  The argument is an adaption of one originally due to R. Richardson (see \cite{birkes71} pgs. 464-465). 
%%%%%%%%%%%%%%%%%%%%%%%%%%%%%%%%%%%%%%%%%%%%%%%%%%%%%%%%%%%%%%%%%%%%%%%%%%%%%%%%%%%%%%%%
\begin{proof} The proof is by contradiction. Therefore we assume that  
\begin{align}
\mathscr{Z}\cap \big(\overline{\mathcal{O}}_w\setminus {\mathcal{O}}_w\big) \neq \emptyset
\end{align}
and for \emph{every} algebraic torus $H$ in $G$ we have 
\begin{align}
\overline{H\cdot[w]}\cap \mathscr{Z}=\emptyset \ .
\end{align}

Fix any maximal algebraic torus $T$ of $G$. We assume that there is a finite collection 
\begin{align}
\mathscr{C}=\{ U_1 , U_2, \dots , U_m\}
\end{align}
of $T$ invariant affine open sets of  $\mathbb{P}(\mathbb{W})$ satisfying
\begin{enumerate}
 \item $\mathscr{Z}\cap U_i\neq \emptyset$ for all $i\in\{1,2,\dots,m\} $ .\\
 \ \\
 \item $\mathscr{Z}\subset \cup_{i=1}^mU_i $ . \\
 \ \\
 \item For all $\kappa\in K$ (a maximal compact of $G$) there is an $i=i(\kappa)$ such that\\
 \begin{align}
 T\kappa\cdot [w]\subset U_{i(\kappa)} \ .
 \end{align}
 \end{enumerate}
 
 The last assumption implies that for every $\kappa\in K$ there is a ball $\kappa\in B_{\delta(\kappa),i(\kappa)}\subset K$ (in any Riemannian metric on $K$ ) such that 
 \begin{align}
 x\cdot[w]\in U_{i(\kappa)} \ \mbox{for all $x\in B_{\delta(\kappa),i(\kappa)}$}\ .
 \end{align}
Furthermore we have that
\begin{align}
 \overline{T\kappa\cdot [w]}\cap U_{i(\kappa)} \cap \mathscr{Z}\cap U_{i(\kappa)}=\emptyset \ .
 \end{align}
 Therefore there exists a $T$ invariant
 \begin{align}
 f=f_{\kappa , i(\kappa)}\in \mathbb{C}[U_{i(\kappa)}]^T \ 
 \end{align}
 satisfying
 \begin{enumerate}
 \item $ f|_{ \overline{T\kappa\cdot [w]}\cap U_{i(\kappa)}}\equiv 1$ . \\
 \ \\
\item $f|_{\mathscr{Z}\cap U_{i(\kappa)}}\equiv 0 $ .
\end{enumerate}
  
 Therefore by choosing $r(\kappa)<\delta(\kappa)$ we may assume that
 \begin{align}
 |f_{\kappa , i(\kappa)}(x\cdot[w])|>0 \ \mbox{for all $x\in B_{r(\kappa),i(\kappa)}$} \ .
 \end{align}
 Now we extract a finite covering  $\{B_{ij}\}$ of $K$ by balls $B_{ij}$ satisfying
 \begin{enumerate}
 \item $x\cdot [w]\in U_{j}$  for all $x\in B_{ij}$ . \\
 \ \\
 \item $|f_{ij}(x\cdot [w])|>0$ for all $x\in B_{ij}$ .
 \end{enumerate}
 Let $\{\varphi_{ij}\}$ be a partition of unity subordinate to this cover. We define a continuous function $F$ on $K$ by
 \begin{align}
 F(x):= \sum_{i,j}\varphi_{ij}(x)|f_{ij}(x\cdot [w])| \ .
 \end{align}

 It is clear from the construction that $F$ is positive. Therefore by compactness $F$ has a positive lower bound on $K$.
 Our assumption at the outset was 
 \begin{align}
 \overline{\mathcal{O}_w}\cap\mathscr{Z}\neq \emptyset \ .
 \end{align}
 The Cartan decomposition
 \begin{align}
 G=KTK
 \end{align}
 and the assumption that $\mathscr{Z}$ is closed and $G$ -invariant imply that there is a sequence
 \begin{align}
\lim_{l\ra \infty} t_l\kappa_l\cdot[w]\in \mathscr{Z}\quad t_l\in T \ , \ \kappa_l\in K .
\end{align}
Then by $T$-invariance of the $f_{ij}$ we have
\begin{align}
\begin{split}
 F(\kappa_l)&=\sum_{i,j}\varphi_{ij}(\kappa_l)|f_{ij}(\kappa_l\cdot [w])| \\
\ \\
&=\sum_{i,j}\varphi_{ij}(\kappa_l)|f_{ij}(t_l\kappa_l\cdot [w])| \ .
\end{split}
\end{align}
Since the $f_{ij}$ vanish on $\mathscr{Z}\cap U_j$  we see that $F(\kappa_l)\ra 0$ as $l\ra \infty$. This contradicts the fact that $F$ has a positive lower bound on $K$ and we are done.
\end{proof}
%%%%%%%%%%%%%%%%%%%%%%%%%%%%%%%%%%%%%%%%%%%%%%%%%%%%%%%%%%%%%%%%%%%%%%%%%%%%%%%%%%%
\subsection{A Kempf-Ness type functional}\label{kempfness} In this section we study semistability in terms of a Kempf-Ness type functional. In fact it is from this point of view that the author arrived at the definition of semistability. As always $(\mathbb{V}, v)$ and $(\mathbb{W}, w)$ are finite dimensional complex rational representations of $G$ together with a pair of nonzero vectors. We equip $\mathbb{V}$ and $\mathbb{W}$ with Hermitian norms . We are interested in the function on $G$  :
 \begin{align}\label{energy}
 G\ni \sigma\ra p_{v,w}(\sigma):=\log||\sigma\cdot w||^2-\log||\sigma\cdot v||^2 \ .
 \end{align}
 Then we have the following fact.
 \begin{proposition}\label{vwlowerbound}
\emph{  $p_{vw}$ is bounded from below on $G$ if and only if $(v,w)$ is semistable.}
 \end{proposition}
 The proposition is a consequence of the following observation.
 \begin{lemma}\label{distance}(Paul \cite{paulcm2012})
 \begin{align}
 p_{v,w}(\sigma)=\log\tan^2d_g(\sigma\cdot [(v,w)] , \sigma\cdot [(v,0)]) \ ,
 \end{align}
 \emph{where $d_g$ denotes the distance in the Fubini-Study metric on} $\mathbb{P}(\mathbb{V}\oplus\mathbb{W})$ .
 \end{lemma}
 
Observe that for any $\sigma,\tau \in G$ we have the inequality
\begin{align}\label{sigma-tau}
d_g(\sigma\cdot [(v,w)] , \sigma\cdot [(v,0)])\leq d_g(\sigma\cdot [(v,w)] , \tau\cdot [(v,0)])\ .
\end{align}
 As a corollary of (\ref{sigma-tau}) and Lemma \ref{distance} we have the much more refined version of Proposition \ref{vwlowerbound}.
\begin{corollary}\label{infi}\emph{ The infimum of the energy of the pair $(v,w)$ is as follows}
 \begin{align}\label{inf}
 \inf_{\sigma\in G}p_{vw}(\sigma)=\log\tan^2d_g(\overline{\mathcal{O}}_{vw},\overline{\mathcal{O}}_{v}) \ .  
 \end{align}  
 \end{corollary}
 Corollary \ref{infi} and Theorem \ref{numericalcriterion}  imply the following result.
 \begin{proposition}\label{energyasymp}
\emph{Assume that
\begin{align}
\inf_{\sigma\in G}p_{vw}(\sigma)=-\infty \ .
\end{align}
Then there exists an algebraic one parameter subgroup $\lambda$ of $G$ such that}
\begin{align}
\lim_{\alpha\ra 0}p_{vw}( {\lambda(\alpha)})= -\infty  \ .
\end{align}
\end{proposition}
\begin{proof}
The proof follows at once from Theorem \ref{numericalcriterion} and the expansion (as $\alpha\ra 0$ )
\begin{align}
p_{vw}(\lambda(\alpha))=(w_{\lambda}(w)-w_{\lambda}(v))\log|\alpha|^2+O(1)\ .
\end{align}  
 \end{proof}
%%%%%%%%%%%%%%%%%%%%%%%%%%%%%%%%%%%%%%%%%%%%%%%%%%%%%%%%%%%%%%%%%%%%%%%%%%%%%%%%%%%%%%%%
\section{Discriminants and Resultants of projective Varieties}
To a nonlinear projective variety $X\ra\cpn$ we associate two polynomials , the $X$-\emph{resultant}  $R_X$,  and the $X$-\emph{hyperdiscriminant} $\Delta_{\xhyp}$ (see \cite{gkz}, \cite{paul2011}, and \cite{paulcm2012}).
 Translation invariance of the Mabuchi energy forces us to {normalize the degrees} of  these polynomials  
\begin{align} 
 R=R(X):= R^{\deg(\Delta_{\xhyp})}_{X} \ ,\ \Delta=\Delta(X):=\Delta ^{\deg(R_X)}_{\xhyp} \ .
  \end{align}
 Below we shall let $r$ denote their \emph{common} degree (where $\mu$ denotes the average of the scalar curvature of $\om$) 
 \begin{align}
 r=\deg(\Delta_{\xhyp})\deg(R_X)=d(n+1)(n(n+1)d-d\mu) \ .
 \end{align}
 
  Given a partition $\beta_{\bull}$ with $N$ parts let $\mathbb{E}_{\beta_{\bull}}$ denote the corresponding irreducible  $G:=\slnc$ module. 
   Let $X\ra\cpn$ be a linearly normal complex projective variety. Then we have
 \begin{align}
\begin{split}
& X  \rightarrow R=R(X) \in \elam\setminus\{0\} \ , \ (n+1)\lambda_{\bull}= \big(\overbrace{ {r} , {r} ,\dots, {r}}^{n+1},\overbrace{0,\dots,0}^{N-n}\big) \ .\\
\ \\
& X \rightarrow \Delta=\Delta(X) \in \emu\setminus\{0\} \ , \  n\mu_{\bull}= \big(\overbrace{ {r} , {r} ,\dots, {r}}^{n },\overbrace{0,\dots,0}^{N+1-n}\big) \ . 
\end{split}
\end{align}
 Moreover, the associations are $G$ equivariant:
\begin{align}
 R(\sigma\cdot X)=\sigma\cdot R(X) \ , \ \Delta(\sigma\cdot X)=\sigma\cdot \Delta(X) \ .
 \end{align}
The irreducible modules $\elam$ and $\emu$ admit the following descriptions
\begin{align}
\begin{split}
&\elam\cong  H^0(\mathbb{G}(N-n-1,N), \mathcal{O}\Big(\frac{r}{n+1}\Big))\cong \mathbb{C}_{r}[M_{(n+1)\times (N+1)}]^{SL(n+1,\mathbb{C})} \\
\ \\
& \emu \cong H^0(\mathbb{G}(N-n ,N), \mathcal{O}\Big(\frac{r}{n}\Big))\cong \mathbb{C}_{r}[M_{n\times (N+1)}]^{SL(n,\mathbb{C})}\ . 
\end{split}
\end{align}
%%%%%%%%%%%%%%%%%%%%%%%%%%%%%%%%%%%%%%%%%%%%%%%%%%%%%%%%%%%%%%%%%%%%%%%%%%%%%%%%%%%%%%%%
\section{K-energy maps and the Kempf-Ness functional of $(R\ ,\ \Delta)$}
 We require the following result recently proved by the author. Given a numerical polynomial $P$
 \begin{align}
 P(T)=c_n\binom{T}{n}+c_{n-1}\binom{T}{n-1}+O(T^{n-2}) \qquad c_n\in \mathbb{Z}_{>0}\  .
 \end{align}
 We define
 \begin{align}
 \mathscr{H}^P_{\cpn}:=\{\mbox{ all (smooth) subvarieties $X\subset \cpn$ with Hilbert polynomial $P$}\} \ .
 \end{align}
 Consider the $G$-equivariant morphisms 
 \begin{align}
 R\ , \ \Delta : \mathscr{H}^P_{\cpn}\ra \mathbb{P}(\elam) \ ,\ \mathbb{P}(\emu) \ .
 \end{align}
In the statement of the next result $h$ denotes some fixed positive Hermitian form on $\cn$. This choice induces Hermitian metrics on $\elam$ and $\emu$. 
\begin{theorem}\label{core} (Paul \cite{paul2011}, \cite{paulcm2012}) 
\emph{ There is a constant $M$ depending only on $c_n$, $c_{n-1}$ and $h$ such that for all $[X]\in \mathscr{H}^P_{\cpn}$ and all $\sigma\in G$ we have}
 \begin{align}
   |(n+1) \nu_{{\om_{FS}}|_{X}}(\sigma)- \frac{1}{c_n^2}
  p_{R(X)\Delta(X)}(\sigma) | \leq M \ .  
 \end{align}
 \end{theorem}
 Theorem \ref{mainthm} follows by applying Proposition \ref{energyasymp} to Theorem \ref{core} .
 
 In the author's view the most important consequence \footnote{ Actually Theorem \ref{mainthm} and this result are equivalent to one another.} of Theorem \ref{mainthm} is the following
 \begin{corollary}\label{polytopes} The Mabuchi energy of $X\ra \cpn$ is uniformly bounded below on $\mathcal{B}$ if and only if
 \begin{align}
 \mathcal{N}(R)\subset \mathcal{N}(\Delta) \quad \mbox{for all maximal algebraic tori $H\leq G$.} \ 
 \end{align}
 \end{corollary}
 \begin{center}\textbf{Acknowledgements}\end{center}
 This work was carried out  during the special trimester in Conformal and K\" ahler geometry at L' Institut Henri Poincar\'e in the Fall of 2012   . The author would like to thank the organizers Matt Gursky, Emmanuel Hebey, and Jeff Viaclovsky for their kind invitation to give a short course on \emph{Canonical K\"ahler metrics and the Stability of Algebraic Varieties}.  It was during these lectures that the author realized he could extend Richardson's argument to the situation relevant for the study of lower bounds on K-energy maps. The author thanks Paul Gauduchon, Christophe Margerin and Nefton Pali for their interest in this work . The author has had the pleasure of discussing this problem over several years with Gang Tian, Xiuxiong Chen, Joel Robbin, and Song Sun, he thanks them for generously sharing their ideas with him. 
 \bibliography{ref}
  \end{document}